\documentclass[12pt, reqno]{amsart}
\usepackage{amsmath, amsthm, amscd, amsfonts, amssymb, graphicx, color, enumerate, enumitem, bm, bbm, todonotes, xfrac}

\usepackage[bookmarksnumbered, colorlinks, plainpages]{hyperref}

\usepackage{tkz-euclide}

\usepackage[capitalize,nameinlink,noabbrev,nosort]{cleveref}
\usepackage{thmtools}
\theoremstyle{plain}
\newtheorem{theorem}{Theorem}[section]
\newtheorem{prop}[theorem]{Proposition}
\newtheorem{lemma}[theorem]{Lemma}

\newtheorem{conjecture}[theorem]{Conjecture}
\theoremstyle{definition}
\newtheorem{definition}[theorem]{Definition}
\newtheorem{example}[theorem]{Example}
\newtheorem{question}[theorem]{Question}
\newtheorem{remark}[theorem]{Remark}
\crefname{definition}{Definition}{Definitions}
\crefname{theorem}{Theorem}{Theorems}
\crefname{lemma}{Lemma}{Lemmas}
\crefname{example}{Example}{Examples}

\renewcommand{\P}{\mathcal{P}}

\newcommand{\E}{\mathcal{E}}
\newcommand{\U}{\mathcal{U}}

\newcommand{\Q}{\mathcal{Q}}

\newcommand{\ZZ}{\mathbb{Z}}
\newcommand{\QQ}{\mathbb{Q}}
\newcommand{\RR}{\mathbb{R}}

\newcommand{\leftb}{\{\!\!\{}
\newcommand{\rightb}{\}\!\!\}}

\DeclareMathOperator{\pre}{pre}

\DeclareMathOperator{\prh}{prh}
\DeclareMathOperator{\prf}{prf}
\DeclareMathOperator{\prs}{prs}

\DeclareMathOperator{\flip}{flip}
\DeclareMathOperator{\supp}{supp}

\DeclareMathOperator{\rev}{rev}
\DeclareMathOperator{\term}{term}
\DeclareMathOperator{\dist}{dist}
\DeclareMathOperator{\lcm}{lcm}

\begin{document}
\setcounter{page}{1}

%------------------------------------------------------------------------------

%Title of the paper
\title[Counterexamples regarding elementary symmetric partitions]{Counterexamples regarding \\ elementary symmetric partitions}

%Author names in alphabetical order
\author[V. Hadelyn, H. Niergarth, W. Li, W. Li]{Vixail Hadelyn, Harper Niergarth, Weiyou Li \MakeLowercase {and} Wenhui Li}

\address{University of Waterloo, Ontario Canada.}
\email{\textcolor[rgb]{0.00,0.00,0.84}{lynHadelyn@gmail.com}}
\address{University of Waterloo, Ontario Canada.}
\email{\textcolor[rgb]{0.00,0.00,0.84}{hniergarth@uwaterloo.ca}}
\address{University of Waterloo, Ontario Canada.}
\email{\textcolor[rgb]{0.00,0.00,0.84}{w43li@uwaterloo.ca}}
\address{University of Waterloo, Ontario Canada.}
\email{\textcolor[rgb]{0.00,0.00,0.84}{w228li@uwaterloo.ca}}

%Abstract, keywords, math subject classification
\begin{abstract}
Ballantine, Beck, and Merca defined the elementary symmetric partition map $\pre_j$ that sends a partition $\lambda$ to a larger partition whose parts are the summands appearing in the evaluation of the $j$-th elementary symmetric polynomial on $\lambda$. They conjectured that $\pre_j$ is injective on the set of partitions of $n$ with length $\ell \geq j$. The $\ell = j$ case was disproved by Devnani and Eyyunni; they instead conjectured the statement to be true for $\ell > j$. In this article, we answer this refined conjecture in the negative by proving that $\pre_j$ is not injective on partitions of $n$ with length $2j$ for $j \geq 3$. We also prove that the analogous map $\prh_j$ defined via the complete homogenous symmetric polynomial is injective on the set of all partitions.  
\end{abstract} \maketitle

%------------------------------------------------------------------------------

\section{Introduction}

Throughout, let $n$ be a finite positive integer. A \textbf{partition} $\lambda = (\lambda_1,\ldots,\lambda_\ell)$ is a weakly decreasing sequence of positive integers. Each $\lambda_i$ is called a \textbf{part} of $\lambda$, the number of parts $\ell = \ell(\lambda)$ is called the \textbf{length} of $\lambda$, and the sum of parts $\lambda_1 + \cdots + \lambda_\ell = |\lambda|$ is called the \textbf{size} of $\lambda$. We identify a partition $\lambda$ with its multiset of parts $\leftb\lambda_i\rightb_{1 \leq i \leq \ell}$.

The \textbf{$j$-th elementary symmetric polynomial} is defined as,
    \begin{align*}
        e_j(x_1,\ldots,x_\ell) = \sum_{1 \leq i_1 < \cdots < i_j \leq \ell} x_{i_1}\cdots x_{i_j},
    \end{align*}
for $1 \leq j \leq \ell$. In \cite{ballantine2025partitions}, Ballantine, Beck, and Merca defined a function $\pre_j$ which sends a partition $\lambda = (\lambda_1,\ldots, \lambda_\ell)$ to the partition whose multiset of parts is,
    \[
        \leftb \lambda_{i_1}\cdots\lambda_{i_j} : 1\le i_1 < i_2 <\cdots<i_j \le \ell\rightb,
    \]
for $1\leq j \leq \ell$. They called the resulting partition $\pre_j(\lambda)$ an \textbf{elementary symmetric partition} since its parts are also the summands appearing in the evaluation $e_j(\lambda_1,\ldots,\lambda_\ell)$. 

For example, take the partition $\lambda = (4,3,3,1)$ and $j=3$. We evaluate $e_3(\lambda)$ as
    \begin{align*}
        e_3(4,3,3,1) = 4\cdot 3 \cdot 3 + 4\cdot 3 \cdot 1 + 4\cdot 3 \cdot 1 + 3\cdot 3 \cdot 1,
    \end{align*}
and so,
    \begin{align*}
        \pre_3(4,3,3,1) = (36,12,12,9).
    \end{align*}

The authors of \cite{ballantine2025partitions} made the following conjecture regarding the function $\pre_j$.

\begin{conjecture}[Conjecture 13 in \cite{ballantine2025partitions}, Conjecture 1 in \cite{ballantine2024elementary}]\label{conj:their conjecture}
    For $j\geq 2$, the function $\pre_j$ is injective on the set of partitions of size $n$ and length $\ell \geq j$.
\end{conjecture}

Li \cite{li2026department} proved the $j=2$ case of \cref{conj:their conjecture}. In fact, they proved a stronger statement where one replaces partitions with weakly decreasing sequences of positive real numbers; \textit{real partitions}. More recently,  Devnani and Eyyunni \cite{devnani2026elementary} disproved the $\ell = j$ case of \cref{conj:their conjecture}, instead conjecturing that the statement is true for $\ell > j$ \cite[Conjecture 1.4]{devnani2026elementary}. In related work \cite{ballantine2026partitions}, the authors considered a class of partitions that is neither a subset nor a superset of partitions of size $n$ and proved that $\pre_j$ is injective on this class. While not answering \cref{conj:their conjecture}, their result provided further evidence towards it.

In this article, we answer the $\ell=2j$ case of \cref{conj:their conjecture} in the negative, also disproving Devnani and Eyyunni's refined conjecture \cite[Conjecture 1.4]{devnani2026elementary}. Our main result is the following theorem.

\begin{theorem}\label{thm:main}
    For each $j\geq 3$, there are infinitely many pairs of distinct partitions $\lambda$ and $\mu$ of length $2j$ such that $|\lambda| = |\mu|$ and $\pre_j(\lambda) = \pre_j(\mu)$.
\end{theorem}

As with the results of Li \cite{li2026department}, our results may also be extended to real partitions (see \cref{rmk:real}); a phenomenon we expect to hold for many results on $\pre_j$. In \cite[Section 4]{ballantine2026partitions}, the authors related elementary symmetric partitions to plethysm, a fundamental operation in the theory of symmetric polynomials. In particular, they pointed out that a consequence of the injectivity of $\pre_2$ \cite{li2026department} is that the family of plethysms $\{e_k[e_2]\}_{k\geq 1}$ determine the original elementary symmetric polynomials $\{e_k\}_{k \geq 1}$ (in finitely many variables $x_1,\ldots,x_\ell$). As a corollary of \cref{thm:main}, we now know that this is \textit{not} true of the family of plethysms $\{e_k[e_j]\}_{k\geq 1}$ for $j \geq 3$. We refer the reader to \cite{macdonald1998symmetric} for more on plethysm.

Our proof of \cref{thm:main} involves defining a map that we call $\flip_j$ (see \cref{def:flip}) that preserves elementary symmetric partitions when $\lambda$ has length $2j$ (see \cref{lemma:flip equality}). We propose the following conjecture regarding the map $\flip_j$.

\begin{conjecture}\label{conj:our conjecture}
    Let $\lambda$ and $\mu$ be distinct partitions of length $\ell$. Then for any $\ell > j \geq 2$, $\pre_j(\lambda) = \pre_j(\mu)$ if and only if $\ell = 2j$ and $\flip_j(\lambda) = \mu$.
\end{conjecture}

Importantly, \cref{conj:our conjecture} would imply that $\pre_j$ is injective on the set of partitions of length $j\leq\ell$ not equal to $j$ or $2j$. This would be strictly stronger than \cref{conj:their conjecture} on partitions not of length $j$ or $2j$ since we do not impose any size requirement.

\begin{remark}
    After completing this article, we learned of similar work by Thomas and Tung \cite{thomas2026injectivitysymmetricpolynomialmaps} that was done independently. They proved that $\pre_3$ is not injective on partitions of $n$ for infinitely many $n$ \cite[Theorem 1.3]{thomas2026injectivitysymmetricpolynomialmaps}. This is equivalent to the $j=3$ case of \cref{thm:main} though our proof methods are distinct. They also proved that $\prh_j$ is injective on the set of all partitions \cite[Theorem 1.5]{thomas2026injectivitysymmetricpolynomialmaps} which we also proved in \cref{thm:homogenous}. Though there are overlaps in our works, we highlight that their results were derived independently from ours.
\end{remark}

%------------------------------------

\section{The map $\flip_j$}\label{sec:counterexample}

In this section, we define a map on partitions that interacts nicely with elementary symmetric partitions. To ensure this map is well-defined, we expand our consideration to \textbf{real-valued partitions}. Let $\P_\ell(\RR)$ denote the set of all weakly decreasing sequences $\lambda = (\lambda_1,\ldots,\lambda_\ell)$ of positive real numbers. More generally, define $\P_\ell(\mathbb{F})$ analogously for $\mathbb{F} \in \{\ZZ,\QQ\}$. Let $c \cdot \lambda = (c\lambda_1,\ldots,c\lambda_{\ell(\lambda)})$

\begin{definition}\label{def:flip}
    For $1 \leq j \leq \ell$, define the map $\flip_j: \mathcal{P}_\ell(\mathbb{R)} \to \mathcal{P}_\ell(\mathbb{R)}$ by setting
    \[
    \flip_j(\lambda_1, \lambda_2, \ldots, \lambda_\ell) =  \sqrt[j]{\lambda_1\lambda_2\cdots \lambda_\ell}\cdot \left(\frac{1}{\lambda_\ell}, \ldots, \frac{1}{\lambda_2},\frac{1}{\lambda_1}\right)
    \]
    for each $\lambda = (\lambda_1,\ldots,\lambda_\ell) \in \P_\ell(\RR)$.
\end{definition}

Note that $\sfrac{1}{\lambda_\ell} \geq \cdots \geq \sfrac{1}{\lambda_2} \geq \sfrac{1}{\lambda_1}$ and so $\flip_j(\lambda)$ is indeed a real-valued partition. Due to our identification of partitions with their multisets of parts, we allow $\flip_j$ to take in multisets as well. In such a case, the order of input/output does not matter.

Before beginning our discussion of $\flip_j$, we introduce helpful notation. For any tuple $\alpha \in \ZZ_{\geq 0}^\ell$, we associate a monomial $x^\alpha = x_1^{\alpha_1}\cdots x_\ell^{\alpha_\ell}$. We let $\lambda^\alpha = \lambda_1^{\alpha_1} \cdots \lambda_\ell^{\alpha_\ell}$ denote the evaluation of $\lambda \in \P_\ell(\RR)$ on the monomial $x^\alpha$. We define the \textbf{support} of a polynomial $f$ as the set $\supp(f) = \{\alpha:[x^\alpha]f \neq 0\}$. Lastly, we let $\E(\ell,j)$ denote $\supp(e_j(x_1,\ldots,x_\ell))$, which consists of every tuple of length $\ell$ that has precisely $j$ ones and zeros everywhere else.  With this notation, we have that $\pre_j(\lambda)$ is the partition whose multiset of parts is $\leftb \lambda^\alpha : \alpha \in \E(\ell,j)\rightb$.

\vspace{4pt}

The following lemma characterizes how $\flip_j$ interacts with $\pre_j$.

\begin{lemma}{\label{lemma:flip equality}}
    Let $\lambda \in \P_\ell(\RR)$. Then, $\pre_j (\flip_j (\lambda)) = \pre_{\ell-j} (\lambda)$. In particular, if $ \ell = 2j$, then $\pre_j (\flip_j (\lambda)) = \pre_j (\lambda)$.
\end{lemma}

\begin{proof}
    For all $1 \leq j \leq \ell$, define the map $f_j : \E(\ell,j) \to \E(\ell, \ell-j)$ by
    \[
    f_j(\alpha_1, \ldots, \alpha_\ell) = (1-\alpha_\ell, \cdots, 1-\alpha_1).
    \]  
    A simple calculation shows that $f_j^{-1} = f_{\ell-j}$ and so $f_j$ is a bijection. 
    
    We claim that $\flip_j(\lambda)^\alpha = \lambda ^ {f_j (\alpha)}$. First observe that the $i$-th coordinate of $\flip_j(\lambda)$ is $(\flip_j(\lambda))_i
    = (\lambda_1\cdots\lambda_\ell)^{1/j}\frac{1}{\lambda_{\ell+1-i}}$.
    Therefore,
    \begin{align}\label{eq:flip 1}
    \flip_j(\lambda)^\alpha
    =
    \prod_{i=1}^\ell
    \left(
    \frac{(\lambda_1\cdots\lambda_\ell)^{1/j}}{\lambda_{\ell+1-i}}
    \right)^{\alpha_i}.
    \end{align}
    Since \(\alpha\in \E(\ell,j)\), we have that \(\sum_{\ell=1}^\ell \alpha_\ell=j\). Hence, \eqref{eq:flip 1} is equal to,
    \begin{align}\label{eq:flip 2}
    (\lambda_1\cdots\lambda_\ell)^{\frac{1}{j}\sum_{\ell=1}^\ell \alpha_\ell}\prod_{i=1}^\ell \lambda_{\ell+1-i}^{-\alpha_i}
    =
    (\lambda_1\cdots\lambda_\ell)\prod_{i=1}^\ell \lambda_{\ell+1-i}^{-\alpha_i}.
    \end{align}
    Now, reindexing the product via $k=\ell+1-i$, \eqref{eq:flip 2} is equal to,
    \[
    (\lambda_1\cdots \lambda_\ell)
    \prod_{k=1}^\ell \lambda_k^{-\alpha_{\ell+1-k}} = \prod_{k=1}^\ell \lambda_k^{\,1-\alpha_{\ell+1-k}} = \lambda^ {f_j (\alpha)},
    \]
    proving the claim. Therefore,
    \begin{align*}
        \leftb \flip_j(\lambda)^\alpha : \alpha \in \E(\ell,j) \rightb = \leftb \lambda^{f_j(\alpha)} : \alpha \in \E(\ell,j) \rightb 
        = \leftb \lambda^\beta : \beta \in \E(\ell,\ell-j) \rightb,
    \end{align*}
    where the last equality comes from the fact that $f_j$ is a bijection. It is then immediate that $\pre_j(\flip_j(\lambda)) = \pre_{\ell-j}(\lambda)$. 
    
    Lastly when $n=2j$, $\pre_j(\flip_j(\lambda)) = \pre_{2j-j}(\lambda) = \pre_j(\lambda)$.
\end{proof}

\begin{example}\label{ex:Wen's counterexample}
    Consider the partition $\lambda = (8, 8, 2, 2, 2, 1)$. Then,
        \begin{align*}
        \flip_3(\lambda) = \sqrt[3]{512}\left(\frac{1}{1},\frac{1}{2},\frac{1}{2},\frac{1}{2},\frac{1}{8},\frac{1}{8}\right) =  (8, 4, 4, 4, 1, 1).
        \end{align*}
    One can check that $\pre_3(\lambda) =\pre_3(\flip_j(\lambda)) = \nu$ where,
    \begin{align*}
        \nu = (128,128,128,64,32,32,32,32,32,32,16,16,16,16,16,16,8,4,4,4).
    \end{align*}
    Note that this pair does not disprove \cref{conj:their conjecture} as $|\flip_3(\lambda)| = 22 \neq 25 = |\lambda|$.
\end{example}

\begin{remark}
    In \cite{cimpoeas2025remarks}, the authors prove that if $\lambda$ and $\mu$ are $d$-ary\footnote{This means that the parts of $\lambda$ are all powers of $d$.} partitions of length $\ell$ (with possibly different sizes), such that $\pre_j(\lambda) = \pre_j(\mu)$ and $\lambda_{i_1}\cdots \lambda_{i_j} = \mu_{i_1}\cdots \mu_{i_j}$ for all $1 \leq i_1 <\cdots < i_j\leq \ell$, then $\lambda = \mu$. The partitions $\lambda$ and $\flip_j(\lambda) = \mu$ in \cref{ex:Wen's counterexample} are examples of distinct $2$-ary (or binary) partitions satisfying $\pre_3(\lambda) = \pre_3(\mu)$ but not the product condition. Indeed, one of the (many) choices of $i_1,i_2,i_3$ that fails is $3,4,5$:
        \begin{align*}
            \lambda_3\lambda_4\lambda_5 = 2\cdot 2\cdot 2 \neq 4 \cdot 4 \cdot 1 = \mu_3\mu_4 \mu_5.
        \end{align*}
    Therefore, the product condition is necessary for their result to hold.
\end{remark}

Our next (more challenging) step is to find partitions $\lambda$ of length $2j$ that have the same size as their image $\flip_j(\lambda)$. A fruitful observation is that this is equivalent to satisfying the algebraic equation,
    \begin{align}\label{eq:size eq}
        \sum_{i=1}^{2j}\lambda_i = \sqrt[j]{\lambda_1\lambda_2 \cdots \lambda_{2j}}\sum_{i=1}^{2j}\frac{1}{\lambda_i}.
    \end{align}

\begin{remark}
    Rearranging terms in \eqref{eq:size eq}, the equation becomes,
        \begin{align}\label{eq:stat}
            \left(\frac{1}{2j} \sum_{i=1}^{2j}\lambda_i\right) \left(\frac{2j}{\sum_{i=1}^{2j}\frac{1}{\lambda_i}}\right)= \left(\sqrt[2j]{\lambda_1\lambda_2\cdots\lambda_{2j}}\right)^2.
        \end{align}
    The reader who has brushed up on their statistics may notice that \eqref{eq:stat} is equivalent to $\text{AM}(\lambda)\text{HM}(\lambda) = \text{GM}(\lambda)^2$ where AM, GM, and HM are the arithmetic, geometric, and harmonic means respectively. This may be of independent interest due to the well known fact that $\text{AM}(\lambda) \geq \text{GM}(\lambda) \geq \text{HM}(\lambda)$. Hence the solutions to \eqref{eq:size eq} that we will construct in \cref{sec:graphs} ``balance'' this chain of inequalities.
\end{remark}

An immediate concern is the $j$-th root in the definition of $\flip_j$ as this may introduce irrational values. For example, $\flip_2(2,1,1,1) = (\sqrt{2},\sqrt{2},\sqrt{2},\frac{\sqrt{2}}{2})$. The way we will assuage this concern is to work with \textbf{normalized partitions}: partitions whose product of parts is one. Of course, this requires us to work with rational partitions as the only normalized integer partition is $(1,\ldots,1)$. Let $\overline{\P_\ell}(\QQ)$ denote the set of all normalized rational partitions of length $\ell$. Luckily, we may translate back to integer partitions with the following map. 

\begin{prop}\label{prop:normalized}
    Define the map $\Psi:\overline{\P_{2j}}(\QQ) \to \P_{2j}(\ZZ)$ by,
        \begin{align*}
            \Psi\left(\frac{p_1}{q_1},\frac{p_2}{q_2}\ldots,\frac{p_{2j}}{q_{2j}}\right) = Z\cdot \left(\frac{p_1}{q_1},\frac{p_2}{q_2}\ldots,\frac{p_{2j}}{q_{2j}}\right) 
        \end{align*} 
        where $Z = \lcm(p_1,p_2,\ldots,p_{2j},q_1,q_2,\ldots,q_{2j})$. Then,
            \begin{enumerate}
                \item[(a)] $\Psi$ is injective.
                \item[(b)] Both $\Psi(\lambda)$ and $\flip_j(\Psi(\lambda))$ are in $\P_{2j}(\ZZ)$.
                \item[(c)] $|\lambda| = |\flip_j(\lambda)|$ implies $|\Psi(\lambda)| = |\flip_j(\Psi(\lambda))|$.
            \end{enumerate}
\end{prop}

\begin{proof}
    For (a), suppose that $\Psi(\lambda) = \Psi(\mu)$, By definition, $Z\cdot\lambda = Z'\cdot\mu$ where $Z,Z'> 0$. Since $\lambda = \frac{Z'}{Z}\cdot\mu$ is normalized, $\left(\frac{Z'}{Z}\right)^{2j} = 1$. Hence, $Z = Z'$, and $\lambda = \mu$. 

    For (b), $\Psi(\lambda) \in \P_{2j}(\ZZ)$ by choice of $Z$. Further, 
    \begin{align*}
        \flip_j\left(\Psi\left(\frac{p_1}{q_1},\frac{p_2}{q_2},\ldots,\frac{p_{2j}}{q_{2j}}\right)\right) &= \flip_j\left(\frac{Zp_1}{q_1},\frac{Zp_2}{q_2},\ldots,\frac{Zp_{2j}}{q_{2j}}\right) \\
        &= \sqrt[j]{\prod_{i=1}^{2j}\frac{Zp_i}{q_i}}\cdot \left(\frac{q_{2j}}{Zp_{2j}},\ldots,\frac{q_{2}}{Zp_{2}},\frac{q_{1}}{Zp_{1}}\right) \\
        &= Z^2 \cdot\left(\frac{q_{2j}}{Zp_{2j}},\ldots,\frac{q_{2}}{Zp_{2}},\frac{q_{1}}{Zp_{1}}\right) \\
        &= Z \cdot\left(\frac{q_{2j}}{p_{2j}},\ldots,\frac{q_{2}}{p_{2}},\frac{q_{1}}{p_{1}}\right),
    \end{align*}
    and this is in $\P_{2j}(\ZZ)$ again by choice of $Z$.
    
    Lastly, for (c), by the work above, $|\flip_j(\Psi(\lambda))| = Z\cdot|\flip_{j}(\lambda)|$. Thus, if $|\lambda| = |\flip_j(\lambda)|$, then,
        \begin{align*}
            |\flip_j(\Psi(\lambda))| = Z\cdot |\flip_{j}(\lambda)| = Z\cdot |\lambda| = |\Psi(\lambda)|.
        \end{align*}

    \begin{example}
        The rational partition $\lambda = \left( \frac{5}{2}, \frac{12}{5},\frac{4}{3},\frac{4}{7},\frac{1}{2},\frac{7}{16}\right)$ is normalized as 
            \begin{align*}
                \frac{5}{2}\cdot \frac{12}{5}\cdot\frac{4}{3}\cdot\frac{4}{7}\cdot\frac{1}{2}\cdot\frac{7}{16} = \frac{6720}{6720} = 1.
            \end{align*}
        Since $\lcm(16,12,7,5,4,3,2,1) = 1680$, 
            \begin{align*}
                \Psi(\lambda) = 1680 \cdot \lambda = (4200,4032,2240,960,840,735).
            \end{align*}
        One may check that,
            \begin{align*}
                \flip_3(\Psi(\lambda)) = (3840,3360,2940,1260,700,672).
            \end{align*}
    \end{example}
\end{proof}

The main goal of the next two sections will be to construct normalized rational partitions $\lambda$ of length $2j$ that satisfy two properties:
    \begin{itemize}
        \item[(P1)] $|\lambda| = |\flip_j(\lambda)|$ or equivalently $\sum_{i=1}^{2j} \lambda_i = \sum_{i=1}^{2j} \frac{1}{\lambda_i}$.
        \item[(P2)] $\lambda \neq \flip_j(\lambda)$ or equivalently $\leftb \lambda_i:1\leq  i \leq 2j \rightb \neq \leftb \frac{1}{\lambda_i}:1\leq  i \leq 2j \rightb$
    \end{itemize}
Note that the ``equivalently'' statements are \textit{not} true unless our partitions are normalized, further motivating this choice.

%------------------------------------

\section{Constructing partitions satisfying (P1)}\label{sec:graphs}

In this section we will construct normalized rational partitions satisfying (P1). This will be done by taking certain directed graphs we call \textit{embeddings} and using the data of their edges to construct the desired partition. Throughout the rest of this article, any union or intersection symbol shall be interpreted as union or an intersection of multisets unless explicitly stated otherwise.

We recall some definitions on graphs. A \textbf{directed graph} is a tuple $G = (V,E)$ where $V$ is a finite set and $E$ is a finite set of ordered tuples $(u,v)$ with $u,v \in V$ and $u\neq v$. We consider $(u,v)$ as a directed edge $u \to v$. Importantly, our definition does not allow for multiedges or loops. 

For a vertex $v \in V$, the \textbf{in-degree} and \textbf{out-degree} are defined respectively as,
    \[
    \delta^-(v) = |\{ e \in E: e = (x,v) \}|, \quad \delta^+(v) = |\{ e \in E: e = (v,x)\}|.
    \]
A directed graph $G = (V,E)$ is \textbf{Eulerian} if $\delta^-(v) = \delta^+(v)$ for all $v \in V$. Lastly, let $\rev(G)$ be the directed graph obtained from $G$ with all edges reversed.

\begin{definition}\label{def:embedding}
    We say a directed graph $G = (V,E)$ is an \textbf{embedding} if $G$ is Eulerian and $V \subset (-\frac{\pi}{4},\frac{\pi}{4})$. More generally, for a set $\U$ dense in $(-\frac{\pi}{4},\frac{\pi}{4})$, we say $G$ is an $\U$-embedding if $V \subset \U \subseteq (-\frac{\pi}{4},\frac{\pi}{4})$.
    
    To each embedding $G$, we associate a multiset $h(G)$ defined as,
           \[h(G) = \bigcup_{(u,v) \in E(G)}\leftb h^+(u,v),h^-(u,v) \rightb \quad \text{where} \quad h^{\pm}(u,v) = \frac{\cos(2u)\cos(u \pm v)}{\cos(2v)\cos(u \mp v)}.\]
    The reader should consider each edge $(u,v)$ as \textit{contributing} the two entries $h^+(u,v)$ and $h^-(u,v)$ to the multiset $h(G)$. Note that an embedding $G$ with $j$ edges will correspond to a multiset $h(G)$ with $2j$ entries, which is inline with our goal. 
\end{definition}

We say an embedding $G$ satisfies (P1) and (P2) if its corresponding multiset $h(G)$ does. We visualize these embeddings by placing (or embedding) the vertices on their corresponding point on the right quarter of the unit circle (see \cref{ex:embeddings}). Before we prove things about $h(G)$, we collect some helpful trigonometric identities.

\begin{lemma}\label{lemma:trig}
    Let $u,v \in \RR$. Then,
    \begin{enumerate}
        \item[(a)] $\cos(u+v) = \frac{1-\tan(u)\tan(v)}{\sec(u)\sec(v)}$,
        \item[(b)] $\cos(u+v)\cos(u-v) = \frac{1}{2}(\cos(2u) + \cos(2v))$,
        \item[(c)] $\cos(u+v)^2 + \cos(u-v)^2 = 1 + \cos(2u)\cos(2v)$.
    \end{enumerate}
\end{lemma}

\begin{proof}
    To prove (a), observe that,
        \begin{align*}
            \cos(u+v) &= \cos(u)\cos(v) - \sin(u)\sin(v) \\
            &= \cos(u)\cos(v)(1-\tan(u)\tan(v)) \\
            &= \frac{1-\tan(u)\tan(v)}{\sec(u)\sec(v)},
        \end{align*}
    where the first equality is the cosine addition identity \cite[pp.~125]{gelfand_trigonometry}.
    
    (b) follows directly from the product-to-sum identity for cosines \cite[pp.~136]{gelfand_trigonometry}. 
    
    Lastly, for (c) we may write,
        \begin{align}\label{eq:cos}
            \cos(u+v)^2 + \cos(u-v)^2 &= \frac{1 + \cos(2u + 2v)}{2} + \frac{1 + \cos(2u - 2v)}{2} \nonumber\\
            &= 1 + \frac{1}{2}\left(\cos(2u+2v) + \cos(2u - 2v)\right),
        \end{align}
    where the first equality is the power reduction identity \cite[pp.~145]{gelfand_trigonometry}. From (b), we have $\cos(2x) + \cos(2y) = 2\cos\left(x+y\right)\cos\left(x-y\right)$. Applying this to \eqref{eq:cos} with $x = u+v$ and $y = u - v$ so that $x+y= 2u$ and $x-y = 2v$ yields,
        \begin{align*}
            1 + \frac{1}{2}\left(2 \cos(2u)\cos(2v)\right) = 1 + \cos(2u)\cos(2v),
        \end{align*}
    as desired.
\end{proof}

We say an embedding $G = (V,E)$ is \textbf{rational} if every $v \in V$ is of the form $\arctan(p)$ for $p \in \QQ$. Note that in order for $\arctan(p) \in (-\frac{\pi}{4},\frac{\pi}{4})$, it suffices to assume $p \in (-1,1)$. Hence, we have the following alternative description.

\begin{lemma}
    The set $\mathcal{Q} = \arctan(\QQ \cap (-1,1))$ is dense in $(-\frac{\pi}{4},\frac{\pi}{4})$. Hence, an embedding $G$ is rational if and only if $G$ is a $\mathcal{Q}$-embedding.
\end{lemma}

\begin{proof}
    As $\QQ$ is dense in $\RR$, $\QQ \cap (-1,1)$ is dense in $(-1,1)$. Since $\arctan$ is continuous, the image of the dense set $\QQ \cap (-1,1)$ must also be dense. The second claim follows from the first.
\end{proof}

For vertices $u,v \in \Q$, $h^\pm(u,v)$ outputs a rational number, justifying the name. This is detailed in the following proposition. 

\begin{prop}\label{prop:rational}
    For $u,v \in \Q$, let $p,q \in \QQ$ be such that $u = \arctan(p)$ and $v = \arctan(q)$. Then,
        \begin{align*}
            h^\pm(u,v) = \frac{1 - p^2}{1 + p^2}\cdot\frac{1 + q^2}{1 - q^2}\cdot
            \frac{1\mp pq}{1\pm pq}.
        \end{align*}
    In particular, $h^\pm(u,v) \in \QQ$.
\end{prop}

\begin{proof}
    By \cref{lemma:trig} (a) and the fact that secant is an even function,
    \begin{align*}
        h^\pm(u,v) &= \frac{1-\tan(u)^2}{1 + \tan(u)^2}\cdot\frac{1 + \tan(v)^2}{1  - \tan(v)^2}\cdot \frac{1\mp\tan(u)\tan(v)}{1\pm\tan(u)\tan(v)}\\
        &= \frac{1\mp p^2}{1\pm p^2}\cdot\frac{1\pm q^2}{1\mp q^2}\cdot
            \frac{1\mp pq}{1\pm pq},
    \end{align*}
    as desired.
\end{proof}

The reader may also take \cref{prop:rational} as a trigonometry free definition of $h(G)$ by letting $p,q \in (-1,1)$. This form is particularly useful for by-hand calculations (see \cref{ex:embeddings}), though the form appearing in \cref{def:embedding} will be more useful for proofs.

We collect some nice properties of the multiset $h(G)$.

\begin{lemma}\label{lemma:nice properties}
    For an embedding $G = (V,E)$ with $j$ edges, 
        \begin{itemize}
            \item[(a)] $h(G)$ is normalized; equivalently $\prod_{x \in h(G)}x = 1$.
            \item[(b)] $h(G) \in \overline{\P_n}(\QQ)$ if $G$ is rational.
            \item[(c)] $\flip_j(h(G)) = h(\rev(G))$
        \end{itemize}
\end{lemma}

\begin{proof}
    We first show (a). Indeed,
        \begin{align*}
            \prod_{x\in h(G)}x &= \prod_{(u,v) \in E} \left(\frac{\cos(2u)\cos(u + v)}{\cos(2v)\cos(u - v)}\right) \cdot \left(\frac{\cos(2u)\cos(u - v)}{\cos(2v)\cos(u + v)}\right) \\
            &= \prod_{(u,v) \in E} \left(\frac{\cos(2u)}{\cos(2v)}\right)^2 \\
            &= \prod_{v \in V} \cos(2v)^{2\delta^+(v)}\cos(2v)^{-2\delta^-(v)} = 1.
        \end{align*}
    The last equality follows from the assumption that $G$ is Eulerian.
    
    Next, (b) follows from (a) and \cref{prop:rational}.
    
    Lastly, for (c) we will show both sides are equal to $\leftb\frac{1}{x}:x \in h(G)\rightb$. The left hand side follows from (a) and the definition of $\flip_j$. The right hand side follows from the fact that $h^\pm(u,v) = \frac{1}{h^\mp(v,u)}$.
\end{proof}

\begin{example}\label{ex:embeddings}
    \begin{figure}[htp!]
        \centering
            \begin{tikzpicture}[>=stealth]
                \def \r {3}; %radius
                \def \p {2}; %vertex size
                \begin{scope}[xshift=0pt]
                
                % angles in degrees
                \def\a{26.565} % arctan(1/2)
                \def\b{8.13} % arctan(1/7)
                \def\c{36.87} % arctan(3/4)
                % halfcircle drawing
                \draw[color=gray,dashed] (90:\r) arc (90:-90:\r);
                \draw[color=gray,dashed] (0,\r) -- (0,-\r);
                \draw[color=gray,dashed] (0,0) -- (\r,0);
                % part of halfcircle containing those nodes
                \draw[color=black,thick] (45:\r) arc (45:-45:\r);
                \draw[color=black,thick] (0,0) -- (45:\r);
                \draw[color=black,thick] (0,0) -- (-45:\r);
                % drawing of graph
                \fill (\a:\r) circle (\p pt);
                \fill (\b:\r) circle (\p pt);
                \fill (\c:\r) circle (\p pt);
                \draw[->,shorten >=\p pt,thick] (\a:\r) to[bend right=20] (\b:\r);
                \draw[->,shorten >=\p pt,thick,] (\b:\r) to[out=   180, in=230, looseness=1.5] (\c:\r);
                \draw[->,shorten >=\p pt,thick] (\c:\r) to[bend right=35] (\a:\r);
                % vertex labels
                \fill (\a:\r) circle (2pt) 
                node[right] {\tiny$\arctan(1/2)$};
                \fill (\b:\r) circle (2pt) 
                node[right] {\tiny$\arctan(1/7)$};
                \fill (\c:\r) circle (2pt) 
                node[right] {\tiny$\arctan(3/4)$};
                \end{scope}
%------------------------------------
                \def \r {3}; %radius
                \def \p {2}; %vertex size
                \begin{scope}[xshift=170pt]

                \def\a{36.87} % arctan(3/4)
                \def\b{18.435} % arctan(1/3)
                \def\c{-11.31} % arctan(1/5)
                \def\d{-39.8} % arctan(-5/6)
                % halfcircle drawing
                \draw[color=gray,dashed] (90:\r) arc (90:-90:\r);
                \draw[color=gray,dashed] (0,\r) -- (0,-\r);
                \draw[color=gray,dashed] (0,0) -- (\r,0);
                % part of halfcircle containing those nodes
                \draw[color=black,thick] (45:\r) arc (45:-45:\r);
                \draw[color=black,thick] (0,0) -- (45:\r);
                \draw[color=black,thick] (0,0) -- (-45:\r);
                % where the flipped nodes would be
                \fill[color=black] (\a:\r) circle (2pt) 
                node[right] {\tiny$\arctan(3/4)$};
                \fill[color=black] (\b:\r) circle (2pt) 
                node[right] {\tiny$\arctan(1/3)$};
                \fill[color=black] (\c:\r) circle (2pt) 
                node[right] {\tiny$\arctan(-1/5)$};
                \fill[color=black] (\d:\r) circle (2pt) 
                node[right] {\tiny$\arctan(-5/6)$};
                \draw[->,shorten >=\p pt,thick] (\d:\r) to[bend left=25] (\c:\r);
                \draw[->,shorten >=\p pt,thick] (\c:\r) to[bend left=25] (\b:\r);
                \draw[->,shorten >=\p pt,thick] (\b:\r) to[bend left=25] (\a:\r);
                \draw[->,shorten >=\p pt,thick] (\a:\r) to[bend right=35] (\d:\r);
                \end{scope}
    \end{tikzpicture}
        \caption{Examples of $\Q$-embeddings.}
        \label{fig:good embeddings}
    \end{figure}
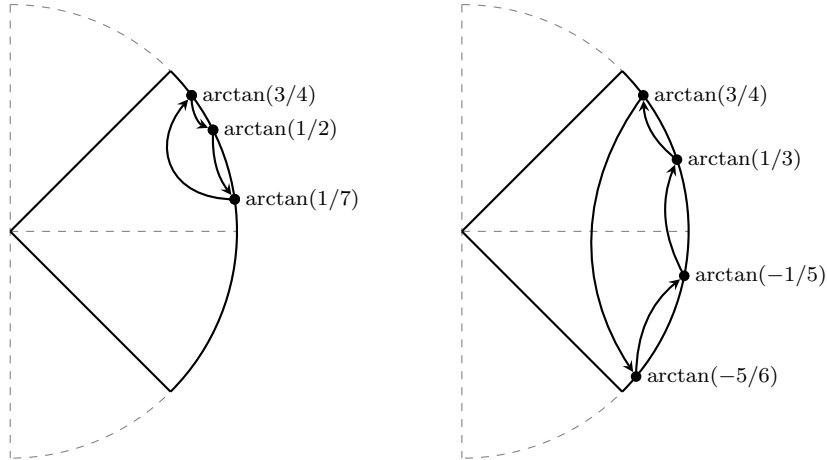               

Consider the embeddings $G$ and $H$ appearing on the left and right in \cref{fig:good embeddings}. They are both $\Q$-embeddings as their vertices are of the form $\arctan(p)$ for $p\in \QQ$. Using the formula from \cref{prop:rational}, we see that the $h^+$ contribution for the edge $\arctan(3/4) \to \arctan(1/2)$ may be computed as follows. 
    \begin{align*}
        h^+(\arctan(3/4),\arctan(1/2)) &= \frac{1 - \left(\frac{3}{4}\right)^2}{1 + \left(\frac{3}{4}\right)^2} \cdot \frac{1 + \left(\frac{1}{2}\right)^2}{1 - \left(\frac{1}{2}\right)^2} \cdot \frac{1- \left(\frac{3}{4}\cdot\frac{1}{2}\right)}{1+ \left(\frac{3}{4}\cdot\frac{1}{2}\right)} \\
        &= \frac{7}{16}\cdot\frac{16}{25}\cdot\frac{5}{4}\cdot \frac{4}{3}\cdot \frac{5}{8} \cdot \frac{8}{11}\\
        &= \frac{7}{33}.
    \end{align*}

Continuing this way, we see that the multisets associated to $G$ and $H$ are,
\[h(G) = \left\{\!\!\!\left\{\frac{7}{33},\frac{77}{75}, \frac{13}{24},\frac{75}{104}, \frac{600}{217},\frac{744}{175} \right\}\!\!\!\right\},\]
and
\[h(H) = \left\{\!\!\!\left\{\frac{1001}{3660},\frac{715}{5124},\frac{105}{104},\frac{120}{91},\frac{12}{7},\frac{100}{21},\frac{5551}{825},\frac{1281}{3575} \right\}\!\!\!\right\}.\]
\end{example}

We implore the reader to confirm that the multisets appearing in \cref{ex:embeddings} satisfy (P1). It is a nice fact that \textit{every} multiset associated to an embedding will satisfy (P1), as we prove in the following lemma.

\begin{lemma}\label{lemma:P1}
    For any embedding $G = (V,E)$ with $j$ edges, $|h(G)| = |\flip_j(h(G))|$.
\end{lemma}

\begin{proof}
    By \cref{lemma:nice properties} (a), the equality $|h(G)| = |\flip_j(h(G))|$ is equivalent to $\sum_{x \in h(G)}x = \sum_{x\in h(G)}\frac{1}{x}$, which we now show. First we write,
        \begin{align*}
            \sum_{x \in h(G)} \left(x - \frac{1}{x}\right) = \sum_{(u,v) \in E}\term(u,v),
        \end{align*}
    where,
        \begin{align*}
            \term(u,v)
            &= \frac{\cos(2u)\cos(u + v)}{\cos(2v)\cos(u - v)} - \frac{\cos(2v)\cos(u - v)}{\cos(2u)\cos(u + v)}\\
            &+ \frac{\cos(2u)\cos(u - v)}{\cos(2v)\cos(u + v)} - \frac{\cos(2v)\cos(u + v)}{\cos(2u)\cos(u - v)}.
        \end{align*}
    Factoring this, we have,
        \begin{align*}
            \term(u,v) = \left(\frac{\cos(2u)}{\cos(2v)} - \frac{\cos(2v)}{\cos(2u)}\right) \cdot \left(\frac{\cos(u+v)}{\cos(u-v)} + \frac{\cos(u-v)}{\cos(v+u)}\right).
        \end{align*}
    Let $A = \cos(2u)$ and $B = \cos(2v)$. With this notation,
        \begin{align}
            \frac{\cos(2u)}{\cos(2v)} - \frac{\cos(2v)}{\cos(2u)} = \frac{A}{B} - \frac{B}{A} = \frac{A^2 - B^2}{AB}.
        \end{align}
    Further, simplifying and applying \cref{lemma:trig} (a) and (b),
        \begin{align*}
            \frac{\cos(u+v)}{\cos(u-v)} + \frac{\cos(u-v)}{\cos(v+u)} = \frac{\cos(u+v)^2 + \cos(u-v)^2 }{\cos(u+v)\cos(u-v)} = 2\cdot\frac{1 + AB}{A+B}.
        \end{align*}
    Therefore, we have the follow convenient form, 
        \begin{align*}
            \term(u,v) = 2\left(\frac{A^2 - B^2}{AB}\right)\left(\frac{1 + AB}{A+B}\right) &= 2\cdot\frac{(A-B)(1+AB)}{AB} \\
            &= 2 \left(A - \frac{1}{A}\right) - 2\left(B-\frac{1}{B}\right).
        \end{align*}
    Letting $F(v) = \cos(2v) - \frac{1}{\cos(2v)}$, $\term(u,v) = 2(F(v) - F(u))$. Reindexing the sum, we have,
        \begin{align*}
            \sum_{(u,v) \in E} \term(u,v) &= 2\cdot\sum_{(u,v) \in E} \left(F(u) - F(v)\right) \\
            &= 2\cdot \sum_{v\in V}(\delta^+(v)F(v) - \delta^-(v)F(v)) = 0.
        \end{align*} The last equality again follows from the assumption that $G$ is Eulerian. This proves the claim.
\end{proof}

%------------------------------------

\section{Finding partitions satisfying (P2)}

In this section, we prove that there are infinitely many normalized rational partitions satisfying (P2) by finding a subset of $\Q$-embeddings that do. We then translate this back to integer partitions and prove \cref{thm:main}. Unlike (P1), infinitely many embeddings fail to satisfy (P2). Indeed, one may convince themselves that any embedding $G = (V,E)$ satisfying the property that $(u,v) \in E$ if and only if $(-v,-u)\in E$ fails (P2). And these are not the only ones; see \cref{ex:bad examples} for another embedding failing (P2) not of this form.

\begin{example} \label{ex:bad examples}
    Consider the embedding $G$ appearing in \cref{fig:bad embedding}. Its associated multiset is,
    \[h(G) = \left\{\!\!\!\left\{\frac{793}{5610},\frac{4758}{935}, \frac{61}{11},\frac{11}{61}, \frac{5610}{793},\frac{935}{4758} \right\}\!\!\!\right\}.\]
    Then $G$ does not satisfy (P2) since $h(G)$ is closed under taking reciprocals of its entries.
    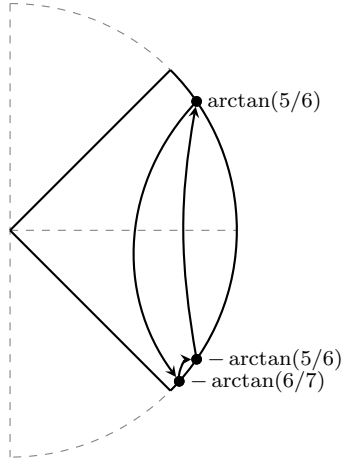
\begin{figure}[htp]
        \centering
            \begin{tikzpicture}[>=stealth]
                \def \r {3}; %radius
                \def \p {2}; %vertex size
                \begin{scope}[xshift = 100pt]
                    %halfcircle drawing
                    \draw[color=gray,dashed] (90:\r) arc (90:-90:\r);
                    \draw[color=gray,dashed] (0,\r) -- (0,-\r);
                    \draw[color=gray,dashed] (0,0) -- (\r,0);
                    %quarter circle drawing
                    \draw[color=black,thick] (45:\r) arc (45:-45:\r);  
                    \draw[color=black,thick] (0,0) -- (45:\r);
                    \draw[color=black,thick] (0,0) -- (-45:\r);
                    %Drawing of graph
                    \foreach \a in {0.6047, -0.6047, -0.7286} {\fill ({\a*180/pi}:\r) circle (\p pt);}
                    \draw[->,shorten >=\p pt,thick] (0.6047*180/pi:\r) to [bend right = 40] (-0.7286*180/pi:\r);
                    \draw[->,shorten >=\p pt,thick] ({-0.7286*180/pi}:\r) to [bend left=50]({-0.6047*180/pi}:\r);
                    \draw[->,shorten >=\p pt,thick] (-0.6047*180/pi:\r) to [bend left = 10] (0.6047*180/pi:\r);
                    %vertex labels
                    \fill (0.6047*180/pi:\r) circle (2pt) node[right] {\tiny$\arctan(5/6)$};
                    \fill (-0.6047*180/pi:\r) circle (2pt) node[right] {\tiny$-\arctan(5/6)$};
                    \fill ({-0.7286*180/pi}:\r) circle (2pt) node[right] {\tiny$-\arctan(6/7)$};
                \end{scope}
            \end{tikzpicture}
        \caption{An example of a $\Q$-embedding not satisfying (P2).}
        \label{fig:bad embedding}
    \end{figure}
\end{example}

In this article, we do not explicitly classify which embeddings satisfy (P2). Instead, we will show that if $G$ is an embedding that fails to satisfy (P2), we may ``wiggle'' the vertices of $G$ to obtain a new embedding $G'$ that does. In fact we will be able to show this for $\U$-embeddings for arbitrary dense sets $\U$. In particular, we may set $\U = \mathcal{Q}$ which lets us find rational $h(G)$ by \cref{prop:rational}. To formalize what we mean, we must first define a notion of \textit{distance} between embeddings.

\begin{definition}
    For two embeddings $G = (V,E)$, $G' = (V',E')$, define $\sigma_{G,G'}:V \to V'$ as the unique order preserving bijection. For example, if $G$ and $G'$ have vertex sets $V = (-\frac{1}{2},0,\frac{1}{3})$ and $V' = (-\frac{2}{9},\frac{1}{3},\frac{1}{2})$, then $\sigma_{G,G'}(-\frac{1}{2}) = -\frac{2}{9}$, $\sigma_{G,G'}(0) = \frac{1}{3}$, and $\sigma_{G,G'}(\frac{1}{3}) = \frac{1}{2}$.
\end{definition}

\begin{definition}\label{def:neighbors}
    We say two embeddings $G = (V,E)$ and $G' = (V',E')$ are \textbf{neighborly} if $(u,v) \in E$ if and only if $(\sigma_{G,G'}(u),\sigma_{G,G'}(v)) \in E'$ for all $u,v \in V$. In other words, $\sigma_{G,G'}$ is an isomorphism of directed graphs. Hence, we denote this by $G \cong_N G'$. For two neighborly embeddings $G = (V,E)$, $G' = (V',E')$, define the \textbf{distance} between them as,
        \begin{align*}
            \dist(G,G') = \sum_{v \in V} |v - \sigma_{G,G'}(v)|.
        \end{align*}
    The reader may think of this as the total displacement needed to wiggle the vertices of $G$ in order to obtain $G'$.
\end{definition}

Any graph isomorphism would work in place of $\sigma_{G,G'}$ in \cref{def:neighbors}. We simply need to choose a canonical one and $\sigma_{G,G'}$ is a natural candidate.

\begin{definition}\label{def:neighborhood}
    For any $\U$-embedding $G = (V,E)$ and $ \epsilon > 0$, define 
    \[\U_\epsilon(G) = \{H: G \cong_N H, \;\text{$H$ is a $\U$-embedding}, \; \;\dist(G,H) < \epsilon\}.\] 
    Further, for a subset of vertices $A\subseteq V$, let $\U_{\epsilon}(G;A)$ be the subset of $\U_\epsilon(G)$ containing all $H$ that differ from $G$ in only $A$.
\end{definition}

Since $\U$ is dense in $(-\frac{\pi}{4},\frac{\pi}{4})$, $\U_\epsilon(G)$ always has infinitely many elements for any choice of $\epsilon$. The next lemma will be key in the following arguments. 

\begin{lemma}\label{lemma:exclusion}
    For any $\U$-embedding $G = (V, E)$, non-empty $A \subseteq V$, and $\epsilon > 0$,
    \begin{align}\label{eq:exclusion}
        \bigcap_{H\in \U_{\epsilon}(G; A)} h(H) = h(G\backslash A).
    \end{align}
\end{lemma}

\begin{proof}
    For convenience, let $I$ and $R$ denote the multisets on the left and right hand sides of \eqref{eq:exclusion}. Note that $I,R \subseteq h(G)$ since $G \in N_\epsilon(G;A)$. Thus it suffices to show that for any $x \in h(G)$, $x$ appears in $I$ and $R$ with the same multiplicity. For notation, we let $m_I(x)$ and $m_R(x)$ denote these multiplicities.
    
    We first prove $m_I(x) \geq m_R(x)$. We say an edge $(u,w)$ \textit{contributes} $x$ if $x \in \leftb h^+(u,w),h^-(u,w)\rightb$. Let $D$ be the set of edges in $E$ contributing $x$. Then $x$ appears with multiplicity $m_R(x)$ in $S = \bigcup_{e\in D}\leftb h^+(e),h^-(e)\rightb$. Since no edge in $D$ is incident to any vertex in $A$, every edge in $D$ appears in every $H \in \U_\epsilon(G;A)$. In other words, for each $H \in \U_\epsilon(G;A)$, $S \subseteq h(H)$ and so the multiplicity of $x$ in $h(H)$ is at least $m_R(x)$. Thus $x$ appears in $I$ with at least multiplicity $m_R(x)$ proving the desired inequality.

    To prove $m_I(x) \leq m_R(x)$, we proceed by induction on $|A|$. The heaviest part of the proof is the basecase where $A = \{v\}$. Let us assume for sake of contradiction that $m_I(R) > m_R(x)$. Define the set $\U_\epsilon(v) = \{v' \in \U:|v-v'| < \epsilon\}$ of vertices $v'$ in $\U$ near $v$. Note that for each choice of $v' \in \U_\epsilon(v)$, there is an associated $H_{v'} \in \U_\epsilon(G)$ obtained by wiggling $v$ to $v'$. Define the set,
    \[
    C = \{u: (u,v) \in E \;\text{ or }\;(v,u) \in E\}
    \]
    of all vertices $u$ in $V$ adjacent to $v$ in $G$. For each $u \in C$, define two sets, 
        \[
        D^\pm_{u} = \{v' \in \U_\epsilon(v): h^\pm(v',u) = x \;\text{ or }\; h^\pm(u,v') = x\},
        \]
    consisting of the choice of wiggled vertex $v'$ such that the edge $e$ connecting $v'$ to $u$ in $H_{v'}$ satisfies $h^\pm(e) =x$. Since we assumed that $m_I(R) > m_R(x)$, we have,
        \begin{align}\label{eq:finite union}
            \U_\epsilon(v) = \bigcup_{u \in C}D_{u}^+\cup D^-_{u}.
        \end{align}
    Here this union is actually a union of sets, not multisets. Indeed, if \eqref{eq:finite union} were not true, then there would exist a vertex $v' \in \U_\epsilon(v)$ such that no edge incident to $v' \in H_{v'}$ contributes $x$. This would imply that $x$ appears with multiplicity $m_R(x)$ in $h(H_{v'})$ and since $H_{v'} \in \U_\epsilon(G)$, $m_I(x) \leq m_R(x)$, contradicting our assumption.

    Now, since the union on the right hand side of \eqref{eq:finite union} is finite while $\U_\epsilon(v)$ is infinite since $\U$ is dense, there exists some $u \in C$ such that $D_u^+$ or $D_u^-$ is infinite. Assume without loss of generality that $D_u^+$ is infinite and that $(u,v) \in G$ (this second assumption is just picking a direction of the original edge). Then defining $f(y) = h^+(u,y)$ as a function of $y$ on the interval $(-\frac{\pi}{4},\frac{\pi}{4})$, we see that $f(y)$ achieves the value $x$ infinitely many times on $\U_\epsilon(v) \subsetneq (-\frac{\pi}{4},\frac{\pi}{4})$. From here, we achieve a contradiction by observing that $f(v)$ is \textit{analytic}\footnote{We do not state the definition of an analytic function here. We refer the reader to \cite{KrantzRealAnalytic}.} on $(-\frac{\pi}{4},\frac{\pi}{4})$ since cosine is. Then by the identity theorem for real analytic functions \cite[Corollary 1.2.6]{KrantzRealAnalytic}, $f(y) = x$ for \textit{all} $y \in (-\frac{\pi}{4},\frac{\pi}{4})$, a clear contradiction.

    Now suppose $|A| > 1$, meaning $A = A'\cup \{v\}$ for $A' \neq \varnothing$ and $v\not\in A'$. Applying the inductive hypothesis and the basecase,
    \begin{align*}
        h(G\backslash A) &=  h((G\backslash A')\backslash v)\\
        &=  \bigcap_{H\in \U_{\epsilon}(G\backslash A'; v)}h(H)\\
        &= \bigcap_{K\in \U_{\epsilon}(G; v)}h(K\backslash A')\\
        &=  \bigcap_{K\in \U_{\epsilon}(G; v)}\left(\bigcap_{P\in \U_{\epsilon}(K; A')}h(P)\right)
    \end{align*}
    The intersection may be rewritten as $\bigcap_{P\in S} h(P)$ where $S = \bigcup_{K \in \U_\epsilon(G;v)}\U_\epsilon(K;A')$. Now, for any $P \in \U_{\epsilon}(G; A)$, let $v' = \sigma_{G, P}(v)$ and let $K$ be obtained from $G$ by replacing $v$ with $v'$. We have that $K \in \U_\epsilon(G;v)$ since $|v-v'| < \epsilon$ and so we may consider $P \in \U_\epsilon(K;A')$. Thus, $P \in S$ and $\U_\epsilon(G;A) \subseteq S$. This implies that $m_I(x) \leq m_R(x)$ for all $x \in h(G)$, as desired.
\end{proof}

\begin{lemma}\label{lemma:density}
    For any $\U$-embedding $G = (V, E)$ with $j$ edges and $\epsilon > 0$, $\U_{\epsilon}(G)$ contains infinitely many $H$ satisfying $h(H) \neq \flip_j(h(H))$.
\end{lemma}

\begin{proof}
    We first prove that $\U_{\epsilon}(G)$ contained at least one such $H$. For contradiction, suppose that every $H\in \U_{\epsilon}(G)$ satisfies $h(H) = \flip_j(h(H))$ and hence, $h(H) = h(\rev(H))$ by \cref{lemma:nice properties} (c). Let $(u,v) \in E$ be an arbitrary edge. Abusing notation, let $(u,v)$ also denote the graph consisting of the single edge $(u,v)$. Applying \cref{lemma:exclusion} with $A = V\setminus\{u,v\}$, we have,
        \begin{align}\label{eq:h of an edge}
            h((u,v)) = h(G/A) = \bigcap_{H \in \U_\epsilon(G;A)}h(H) = \bigcap_{H \in \U_\epsilon(G;A)}h(\rev(H)).
        \end{align}
    The last equality follows from the fact that $\U_\epsilon(G;A) \subset \U_\epsilon(G)$. Since $H$ and $\rev(H)$ are on the same vertex set, \eqref{eq:h of an edge} equals,
        \begin{align*}
            \bigcap_{H \in \U_\epsilon(\rev(G);A)}h(H) = h(\rev(G)/A) = h(\rev((u,v))) = h((v,u)).
        \end{align*}
    Importantly, this allows us to conclude $\leftb h^+(u,v),h^-(u,v)\rightb = \leftb h^+(v,u),h^-(v,u)\rightb$. We claim that this implies $u = \pm v$. First assume $h^+(u,v) = h^+(v,u)$. Then,
        \begin{align*}
            \frac{\cos(2u)\cos(u+v)}{\cos(2v)\cos(u-v)} = \frac{\cos(2v)\cos(u+v)}{\cos(2u)\cos(u-v)} \Longrightarrow \cos(2u)^2 = \cos(2v)^2.
        \end{align*}
    Since $u,v \in (-\frac{\pi}{4},\frac{\pi}{4})$, $\cos(2u),\cos(2v)\geq 0$. Thus, $\cos(2u) = \cos(2v)$ which implies $u = \pm v$.\\
    
    Now assume $h^+(u,v) = h^-(v,u)$. Then we must also have $h^-(u,v) = h^+(v,u)$. Since $h^+(x,y)h^{-}(y,x) = 1$ and $h^{\pm}(x,y) > 0$ for any $x,y \in (-\frac{\pi}{4},\frac{\pi}{4})$ then $1 = h^+(u,v)=h^{-}(u,v)$. This means the first case applies and so $u=\pm v$. By definition, $G$ has no loops or multiedges and so $(u,v)$ is the only edge incident to both $u$ and $v$. But this contradicts that $G$ is Eulerian.

    To prove that $\U_\epsilon(G)$ contains infinitely many $H$ satisfying $h(H) \neq \flip_j(h(H))$, call such $H$ good and assume there are finitely many good $H$'s. Let $\delta$ denote $\min\{\dist(G,H):H \text{ is good}\}$. Then pick $\delta>\epsilon'>0$, and apply the above result to see that there exists a good $H' \in \U_{\epsilon'}(G)$, contradicting minimality of $\delta$.
\end{proof}

Let $\U^{\text{good}}_\epsilon(G)$ denote the subset of $\U_\epsilon(G)$ consisting of $\U$-embeddings $H$ satsifying (P2). Applying \cref{lemma:density} allows us to conclude that this set is infinite. The last thing we must confirm before proving \cref{thm:main} is that the set of multisets we get from these good embeddings (i.e. $\{h(H):H \in \U^{\text{good}}_\epsilon(G)\}$) is also infinite. We collect one useful lemma before proving this.

\begin{lemma}\label{lemma:finite}
    For fixed $A,B > 0$, there are only finitely many $u,v \in (-\frac{\pi}{4},\frac{\pi}{4})$ satisfying $\leftb A,B \rightb = \leftb h^+(u,v),h^-(u,v) \rightb$.
\end{lemma}

\begin{proof}
    We begin by noting that
        \begin{align}
            AB = h^+(u,v)h^-(u,v) = \left(\frac{\cos(2u)}{\cos(2v)}\right)^2 &\Longrightarrow \cos(2u)^2 = AB\cos(2v)^2 \nonumber\\
            &\Longrightarrow \cos(2u) = \sqrt{AB}\cos(2v) \label{eq:trig product}
        \end{align}
    since $u,v \in (-\frac{\pi}{4},\frac{\pi}{4})$. Assume without loss of generality that,
        \begin{align}
            \frac{A}{B} = \frac{h^+(u,v)}{h^-(u,v)} = \left(\frac{\cos(u+v)}{\cos(u-v)}\right)^2 \Longrightarrow \cos(u+v)^2 = \frac{A}{B}\cos(u-v)^2. \label{eq:trig ratio}
        \end{align}
        
    By the power reduction identity \cite[pp.~145]{gelfand_trigonometry} and the cosine addition identity \cite[pp.~125]{gelfand_trigonometry},
        \begin{align}
            \cos(u\pm v)^2 = \frac{1 + \cos(2u \pm 2v)}{2} =  \frac{1+\cos(2u)\cos(2v) \pm \sin(2u)\sin(2u)}{2}\label{eq:cos double}
        \end{align}
    Applying \eqref{eq:cos double} to \eqref{eq:trig ratio},
        \begin{align}
            1+\cos(2u)\cos(2v) + \sin(2u)\sin(2u) = \frac{A}{B}\left(1+\cos(2u)\cos(2v) - \sin(2u)\sin(2u)\right)\nonumber \\
            \Longrightarrow \left(1-\frac{A}{B}\right)(1 + \cos(2u)\cos(2v)) = \left(-1-\frac{A}{B}\right)\sin(2u)\sin(2v)\label{eq:almost poly}
        \end{align}
    Now, we square both sides of \eqref{eq:almost poly} and use fact that $\sin(2x)^2 = 1 - \cos(2x)^2$ to obtain,
        \begin{align*}
            \left(1-\frac{A}{B}\right)^2(1+\cos(2u)\cos(2v))^2 + \left(1+\frac{A}{B}\right)^2(1-\cos(2u)^2)(1-\cos(2v)^2) = 0.
        \end{align*}   
    Finally, set $x = \cos(2u)$ and note that by \eqref{eq:trig product}, $\cos(2v) = \sqrt{AB}x$.
        \begin{align*}
            \left(1 - \frac{A}{B}\right)^2\left(1+\sqrt{AB}x\right)^2 + \left(1+\frac{A}{B}\right)^2 (1-x^2)(1-ABx^2) = 0.
        \end{align*}
    Thus, $x = \cos(2u)$ is the root of a quartic and so can take on at most four values. Since $u \in (-\frac{\pi}{4},\frac{\pi}{4})$, this means $u$ can take on at most eight values. And by \eqref{eq:trig product}, $u$ determines $v$ up to a sign.
\end{proof}

\begin{lemma}\label{lemma:infinite multisets}
    For any $\U$-embedding $G$ and $\epsilon > 0$, the set $\{h(H) : H \in \U^{\textup{good}}_\epsilon(G)\}$ is infinite.
\end{lemma}

Assume for sake of contradiction that the set is finite. For each $\U$-embedding $H$, let $\{h(H)\}$ denote the \textit{set} of numbers appearing in $h(H)$. Let $D = \bigcup_{H \in \U_\epsilon^\text{good}}\{h(H)\}$ where the union is a union of sets. For each pair of (possibly equal) numbers $(A,B) \in D^2$, let $E_{A,B}$ be the set of edges $(u,v)$ satisfying $\leftb A,B\rightb = \leftb h^+(u,v),h^-(u,v)\rightb$ and let $E_{D} = \bigcup_{(A,B) \in D^2}E_{A,B}$ be the set of all edges that could possibly contribute numbers appearing in $D$. By assumption, $D$ is finite and by \cref{lemma:finite} $E_{A,B}$ is finite so $E_D$ is finite as well. Choose an embedding $H \in \U_\epsilon^\text{good}(G)$ that avoids using edges in $E_{D}$, which is possible since $E_D$ is finite. Next, let $V$ be the vertex set of $G$ and let $\delta > 0$ be small enough so that $\U_\delta(H) \subseteq \U_\epsilon(G)$ and no embedding $K \in \U_\delta(H)$ uses an edge in $E_D$, again possible since $E_D$ is finite. Then, by \cref{lemma:density}, there is an embedding $K \in \U_\delta(H) \subseteq \U_{\epsilon}(G)$ such that $h(K)$ satisfies (P2). Hence $K \in \U_\epsilon^\text{good}(G)$ but by construction, $\{h(K)\} \not\subseteq D$ since no edge of $K$ appears in $E_D$. This contradicts the definition of $D$.

\begin{proof}[Proof of \cref{thm:main}]
    For $j \geq 3$, fix some $\Q$-embedding $G$ with $j$ edges and some $\epsilon >0$. Applying \cref{lemma:infinite multisets} with $\U = \Q$, the set $S = \{h(H) :H \in \U_\epsilon^{\text{good}}\}$ is infinite. By \cref{lemma:nice properties} (b), \cref{lemma:P1}, and the definition of $\U_\epsilon^{\text{good}}$, every multiset in $S$ corresponds to a normalized rational partition $\lambda$ that satisfies (P1) and (P2). Define the set $\text{pairs}_j = \{(\Psi(\lambda),\flip_j(\Psi(\lambda)):\lambda \in S\}$ and note the following for each pair $(\Psi(\lambda),\flip_j(\Psi(\mu))) \in \text{pairs}_j$:
        \begin{itemize}
            \item Both $\Psi(\lambda)$ and $\flip_j(\Psi(\lambda))$ are in $\P_{2j}(\ZZ)$ by \cref{prop:normalized} (b).
            \item $|\Psi(\lambda)| = |\flip_j(\Psi(\mu))|$ since $\lambda$ satisfies (P1) and we may apply \cref{prop:normalized} (c).
            \item $\Psi(\lambda) \neq \flip_j(\Psi(\lambda))$ since $\lambda$ satisfies (P2) and $\Psi$ is injective by \cref{prop:normalized} (a).
            \item $\pre_j(\Psi(\lambda)) = \pre_j(\flip_j(\Psi(\lambda)))$ by \cref{lemma:flip equality}.
        \end{itemize}
    We conclude that every pair in $\text{pairs}_j$ is of the form claimed to exist in \cref{thm:main}. And since $\text{pairs}_j$ in infinite by nature of $S$ being infinite, this proves \cref{thm:main}.
\end{proof}

While our proof of \cref{thm:main} is not constructive, our results do give a recipe to construct such partitions. Indeed, one may pick a $\Q$-embedding $G$ at random and if $h(G)$ satisfies (P2), send it through the map $\Psi$. The only point of failure would be if $h(G)$ does not satisfy (P2). In practice, this is highly unlikely. We ran Python code to randomly generate $1,000,000$ $\Q$-embeddings of the 3-cycle and only three had multisets that did not satisfy (P2). We also randomly generated $500,000$ $\Q$-embeddings of the 4-cycle and every single one had multisets that satisfied (P2). We encourage the reader to try to construct their own examples.

We conclude this section by explicitly computing counterexamples and making a few remarks.

\begin{example}\label{ex:final}
    Observe that the $\Q$-embeddings $G$ and $H$ appearing in \cref{ex:embeddings} both satisfy (P2). Let $\lambda = \Psi(h(G))$ and $\mu = \Psi(h(H))$. Then by the proof of \cref{thm:main}, the pairs $(\lambda, \flip_3(\lambda))$ and $(\mu, \flip_4(\mu))$ are counterexamples to the injectivity of $\pre_3$ and $\pre_4$ respectively. Explicitly, these pairs of partitions are below.\\[-30pt]
    
    \[\hspace{0.42in}\lambda= ( 79155648, 51480000, 19115096, 13426875, 10085075, 3949400),\]\\[-33pt]
    \[\flip_3(\lambda)=( 87773400,34372800,25817792,18135000,6733727,4379375),\]
    {\resizebox{5.7 in}{!}{\hspace{0.42in}$\mu=(246508808,174460000,62805600,48312000,36988875,13127688,10020010,5112250)$}}\\
    {\resizebox{5.7 in}{!}{$\flip_4(\mu)=(262553760,133956000,102245000,36287680,27782755,21371350,7693686,5445000)$}}\\[10pt]
    We can see that $|\lambda| = |\flip_3(\lambda)| = 177212094$ and $|\mu| = |\flip_4(\mu)| = 597335231$. Moreover, $\pre_3(\lambda) = \pre_3(\flip_3(\lambda))$ and $\pre_4(\mu) = \pre_4(\flip_4(\mu))$ by \cref{lemma:flip equality}. This may also be checked by a computer.
\end{example}

\begin{remark}\label{rmk:real}
    Recall that in \cite{li2026department}, it was proved that $\pre_2$ is injective on the set of partitions of $n$, and further that this may be extended to real partitions of $n$. Our construction agrees with these results. There are no embeddings with two edges and so our construction fails for $j=2$. Moreover, the proof of \cref{thm:main} may been adjusted to accommodate real partitions simply by choosing $\U = (-\frac{\pi}{4},\frac{\pi}{4})$ and defining a slightly different injective map analogous to $\Psi$.
\end{remark}

\begin{remark}
    While looking for counterexamples to \cref{conj:their conjecture} and before working out our embedding framework, we found the following particularly small counterexample via code,
    \[\lambda = (90,50,45,25,12,12), \quad \flip_3(\lambda) = (75,75,36,20,18,10).\] 
    These partitions are each others image under $\flip_3$ and so $\pre_3(\lambda) = \pre_3(\mu)$ by \cref{lemma:flip equality}. Further, they are both partitions of size $234$ meaning the pair is indeed a counterexample. Once we wrote up this section, we made an attempt to find a rational embedding $G$ such that $\Psi(h(G)) = \lambda$ but failed to do so. Since our methods tend to produce large partitions as can be seen in \cref{ex:final}, it would be interesting to the authors to find such an embedding or prove that no such embedding exists.
\end{remark}

%------------------------------------

\section{Further generalizations}

We conclude with a discussion of generalizations of this injectivity question. In \cite{ballantine2025partitions}, it was suggested that one could evaluate other symmetric polynomials, such as the monomial, complete homogeneous, and Schur polynomials, on any partition $\lambda$ and define a partition with the summands as its parts. 

For example, if we used the complete homogeneous symmetric polynomials $h_j(x_1,\ldots,x_\ell)$, we could define $\prh_j(\lambda)$ as the partition whose multiset of parts is
    \begin{align*}
        \leftb\lambda_{i_1}\cdots\lambda_{i_j}: 1 \leq i_1 \leq \cdots \leq i_j \leq \ell \rightb
    \end{align*}
It turns out that the injectivity question is much easier for $\prh_j$ than for $\pre_j$.

\begin{theorem}\label{thm:homogenous}
    For $j \geq 1$, the function $\prh_j$ is injective on the set of all partitions.
\end{theorem}

\begin{proof}
    Let $\lambda$ and $\mu$ be partitions such that $\prh_j(\lambda) = \prh_j(\mu)$. By definition $\ell(\lambda) = \ell(\mu)$, call this length $\ell$. It suffices to show that $\lambda_i = \mu_i$ for all $1 \leq i \leq \ell$. We proceed by induction. For the basecase, observe that,
        \begin{align*}
            \lambda_1 = \sqrt[j]{\max(\prh_j(\lambda))} = \sqrt[j]{\max(\prh_j(\mu))} = \mu_1.
        \end{align*}
    Assume that $\lambda_i = \mu_i$ for all $1 \leq i \leq k < \ell$. This implies that $\prh_j(\lambda^-) = \prh_j(\mu^-)$ where $\lambda^- = (\lambda_1,\ldots,\lambda_k)$ and $\mu^- = (\mu_1,\ldots,\mu_k)$. Next, notice that
        \begin{align*}
            \lambda_1^{j-1}\lambda_{k+1} = \max(\prh_j(\lambda)\setminus\prh_j(\lambda^-)) = \max(\prh_j(\mu)\setminus\prh_j(\mu^-)) = \mu^{j-1}\mu_{k+1}.
        \end{align*}
    Since $\lambda_1 = \mu_1$, it follows that $\lambda_{k+1} = \mu_{k+1}$, completing the induction.
\end{proof}

It is well known that the elementary and complete homogeneous symmetric polynomials have simple expansions in the monomial symmetric polynomial basis (we refer the reader to \cite{macdonald1998symmetric} for details):
    \begin{align*}
        e_j = m_{1^j}\quad \text{and} \quad h_j = \sum_{\lambda: |\lambda| = j}m_\lambda.
    \end{align*}
In this way, both are examples of homogeneous symmetric polynomials in which the coefficient of every monomial is zero or one; we call such symmetric polynomials \textbf{simple}. Since $\pre_j$ and $\prh_j$ have distinct behavior with respect to injectivity, it may be interesting to determine which simple symmetric polynomials give us injective maps analogous to $\pre_j$ and $\prh_j$.

\begin{question}
    Let $f$ be a simple symmetric polynomial. Define $\prf(\lambda)$ as the partition whose multiset of parts is $\leftb \lambda^\alpha : \alpha \in \supp(f)\rightb$. For which $f$ is $\prf$ injective on the set of partitions? What about on the set of partitions of size $n$?
\end{question}

Lastly, we consider the Schur polynomials $s_\mu(x_1,\ldots,x_\ell)$. They are the generating functions for semistandard young tableaux; fillings of the diagram of $\mu$ with numbers $\{1,\ldots,\ell\}$ such that the rows weakly increase and the columns strictly increase. We denote the set of these as $\text{SSYT}(\mu,\ell)$ and refer the reader to \cite{macdonald1998symmetric} for a more formal definition. Then, we could define $\prs_\mu(\lambda)$ as the partition whose multiset of parts is,
    \begin{align*}
        \leftb \lambda_1^{\#\text{ of 1's in }T}\cdots\lambda_\ell^{\#\text{ of $\ell$'s in }T} : T\in \text{SSYT}(\mu,\ell)\rightb.
    \end{align*}
Another well known fact is that the elementary symmetric polynomials and the complete homogeneous polynomials are special cases of Schur polynomials,
    \[
        s_{(1^j)} = e_j \quad\text{and}\quad s_{(j)} = h_j.
    \]
See \cite{macdonald1998symmetric}. Thus, $\prs_{1^j}$ is \textit{not} injective by \cref{thm:main} while $\prs_{j}$ is by \cref{thm:homogenous}.\vspace{-10pt}
    \begin{question}\label{question:schur}
        For which partitions $\mu$ is $\prs_\mu$ injective on the set of partitions? What about on the set of partitions of size $n$?
    \end{question}

Partial progress has already been made on \cref{question:schur}; Thomas and Tung have proven $\prs_{(2,1)}$ is injective on the set of partitions of $n$ \cite[Theorem 1.7]{thomas2026injectivitysymmetricpolynomialmaps} using techniques similar to Li in \cite{li2026department}.

\section*{Acknowledgments}
Much of this research was conducted through the Directed Research Program (DRP) organized by the Women in Mathematics (WiM) chapter at the University of Waterloo. We are grateful for their support. The second author was partially supported by NSERC grant RGPIN-2021-02568.
%-----------------------------------

\bibliographystyle{amsplain}
\bibliography{references}

\end{document}